\documentclass{article}

\usepackage{amsthm,amssymb,latexsym,amsmath}
\usepackage[all]{xypic}
\usepackage{dsfont}
\usepackage{color}

\usepackage[symbol]{footmisc}

\usepackage{authblk}

\newcommand{\catA}[1]{{\mathfrak A}}
\newcommand{\catI}[1]{{\mathfrak I}}
\newcommand{\catS}[1]{{\mathfrak S}}

\newcommand{\p}[1]{{\mathbb{P}^{#1}}}

\newcommand{\bp}[1]{\widetilde{\mathbb{P}^{#1}}}

\newcommand{\Quot}{{\rm Quot}}

\DeclareMathOperator{\ext}{Ext}
\DeclareMathOperator{\Hom}{Hom}

\DeclareMathOperator{\ho}{H}

\DeclareMathOperator{\aut}{Aut}

\DeclareMathOperator{\obj}{Obj}

\newtheorem{theorem}{Theorem}[section]

\newtheorem{proposition}[theorem]{Proposition}
\newtheorem{lemma}[theorem]{Lemma}
\newtheorem{corollary}[theorem]{Corollary}

\newtheorem{definition}[theorem]{{\bf Definition}}

\title{Framed instanton pairs on the blow-up of the projective $3$-space at a point}

\author[$\ddag$]{Abdelmoubine Amar Henni$^{\dagger}$} 

\affil[$\ddag$]{{\small Departamento de Matem\'atica MTM - UFSC 

Campus Universit\'ario Trindade CEP 88.040-900  

Florian\'opolis-SC, Brazil. 
}}

\date{}

\begin{document}

\maketitle

\vspace{1cm}

\begin{abstract}
We study Huybrechts and Lehn framed sheaves on the Fano $3$-fold given by blowing-up the projective $3$ space at a point. In contrast with the cases of curves and surfaces, there are very few examples of framed sheaves in higher dimensions and the following notes are intended for constructing new examples of such objects in dimension $3.$ The  given examples are shown to come in families and we prove that their moduli space is fine, quasi-projective and unobstructed.

\end{abstract}

\vspace{1cm}

\tableofcontents

\footnotetext[2]{The author was partially supported by the CAPES-COFECUB $08/ 2018$ project and MATH-AMSUD project: GS$\&$MS $21$-MATH$-06$  $2021\hspace{0.1cm}\& \hspace{0.1cm}2022.$
}

\section{Introduction}


In \cite{Huy1, Huy2, Lehn}, the notion of {\em framed pairs} was introduced and the construction of their moduli has been studied. Systematic examination of these objects and the behavior of their moduli spaces were mostly studied on curves and surfaces \cite{Nevins, bruzzo, sala}. In higher dimensions, there are very few examples that have been studied. This is because, some important results and constructions established in the case of curves and surfaces become very hard to extend on general ground, for higher dimensional varieties. An example of such a result is \cite[Theorem 3.1]{bruzzo}; basically framed sheaves defined in \cite{Lehn}, with certain framing, called {\em good framing}, are automatically slope stable as framed pairs when we work on curves or surfaces, but is is not clear under which additional conditions they might automatically enjoy slope stability, in higher dimensions, as showed by Oprea in the example given in \cite[p. 325]{Oprea}.    

More generally, decorated sheaves (and complexes), that is, sheaves (or complexes) with additional structures are widely studied in algebraic geometry due to their abundance and natural occurrence such as Higgs bundles and instantons that appeared first in gauge theory and mathematical physics. Other examples are encountered imposing additional conditions in order to get well behaved moduli problems. For more details the reader might consult \cite{sala} and references therein.

This work is dedicated to the study certain framed sheaves that are related to instantons on the blow-up of the projective space $\p3$ at a point. We show that some of these instantons can be endowed with a framing given by a homomorphism to a fixed instanton sheaf $\mathcal{E}_{D}$ on a divisor $D$ given by the pull-back of a hyperplane in $\p3.$ We prove that such framed sheaves are slope stable as Huybrechts and Lehn pairs. Moreover, that there exists a quasi-projective scheme that represent their moduli functor. In addition, we also show that this is unobstracted for t'Hooft instanton bundles, constructed in \cite{ccgm}.

Other illustrations of moduli spaces of framed sheaves in higher dimensions can be found in the following: 
In \cite{Oprea} Oprea studied the obstruction theory of the moduli of framed sheaves over certain local Calabi-Yau $3$-folds and their topological invariants by using toric techniques and virtual localization. Also, in \cite{ricolfi} the authors proved an isomorphism between the moduli space of certain framed sheaves and the Grothendieck's Quot scheme of points over projective spaces of dimension greater than or equal to $3.$

Our paper is organized as follows; the first section is a reminder and an exposition of the main definitions and results concerning framed sheaves (or pairs), mainly from \cite{Huy1, Huy2}. Section \ref{framedp3} is divided into two parts; the first one contains a background material on instanton sheaves on $\bp3.$ Finally, in the last part of Section \ref{framedp3} we prove our main results.

\section{Framed sheaves and their moduli}

In this section we give some outline of the general theory of framed sheaves and their moduli spaces following references \cite{Huy1,Huy2}.

Let $X$ be a non-singular $d-$dimensional projective variety over an algebraically closed field $k$ endowed with a very ample line bundle $\mathcal{O}_{X}(1).$ The degree of $X$ with respect to the embedding given by $\mathcal{O}_{X}(1)$ is denoted by $g.$ Let $\mathcal{D}$ be a coherent sheaf of $\mathcal{O}_{X}-$modules  and $\delta$ a polynomial with rational coefficients and positive leading term.
\begin{definition}
A  framed sheaf is a pair consisting of a rank $r$ coherent sheaf, of $\mathcal{O}_{X}-$modules, $\mathcal{E}$ and a homomorphism $\alpha:\mathcal{E}\to\mathcal{D},$ called the framing of $\mathcal{E}.$
\end{definition}
We also introduce the function $\epsilon(\alpha)$ defined by $$\epsilon(\alpha)=\left\{\begin{array}{ll}1 & \textnormal{if } \alpha\neq0\\0 &\textnormal{ otherwise } \end{array}\right.$$ We denote by $P_{\mathcal{E}}(n):=\chi(\mathcal{E}(n))$ the \emph{Hilbert polynomial} of $\mathcal{E},$ that is, the Euler characteristic of $\mathcal{E}(n):=\mathcal{E}\otimes\mathcal{O}_{X}(n)$ and we define \emph{Hilbert polynomial of the pair}  as $P_{(\mathcal{E},\alpha)}:=P_{\mathcal{E}}-\epsilon(\alpha)\delta$ the $(\mathcal{E},\alpha).$ The \emph{reduced Hilbert polynomial of the pair} $(\mathcal{E},\alpha)$ is defined by $$p_{(\mathcal{E},\alpha)}:=\frac{P_{(\mathcal{E},\alpha)}}{r}.$$

If $\mathcal{E}'$ is a coherent subsheaf of $\mathcal{E}$ with quotient $\mathcal{E}''=\mathcal{E}/\mathcal{E}',$ then a framing $\alpha:\mathcal{E}\to\mathcal{D}$ induces a framing $\alpha'=\alpha|_{\mathcal{E}'}$ on $\mathcal{E}'$ and $\alpha''=\alpha|_{\mathcal{E}''}$ on $\mathcal{E}'':$ $\alpha''$ is zero if $\alpha'\neq0$ and is the induced homomorphism on $\mathcal{E}''$ if $\alpha'$ vanishes. As for sheaves, the Hilbert polynomial of framed sheaves also behaves additively $$P_{(\mathcal{E},\alpha)}=P_{(\mathcal{E}',\alpha')}+P_{(\mathcal{E}'',\alpha'')}.$$

\begin{definition}
A framed sheaf $(\mathcal{E},\alpha)$ of rank $r$ is said to be (semi-)stable with respect to $\delta$ with Hilbert polynomial $p$ if for all proper subsheaves $\mathcal{E}',$ of rank $r',$ with induced framing $\alpha'$  the inequality $$p_{(\mathcal{E}',\alpha')}\hspace{0,1cm}(\leq)\hspace{0,1cm}p_{(\mathcal{E},\alpha)}.$$
\end{definition}

By $(\leq)$ we mean $p_{(\mathcal{E}',\alpha')}\leq\hspace{0,1cm} p_{(\mathcal{E},\alpha)}$ for semi-stability or $p_{(\mathcal{E}',\alpha')}<\hspace{0,1cm}p_{(\mathcal{E},\alpha)}$ for stability. 

The following results follow directly from the previous definitions:
\begin{lemma}
If $(\mathcal{E},\alpha)$ is semi-stable, then its kernel $\ker\alpha$ is torsion-free, i.e., $\alpha$ embeds the torsion sheaf $T(\mathcal{E})$ of $\mathcal{E}$ as a subsheaf of $\mathcal{D}.$
\end{lemma}
\begin{proof}
If $T(\mathcal{E})$ is the torsion part of the kernel of $\alpha,$ then in the inequality of the definition above $r'=0$ and $\alpha'=0,$ so that $P_{(T(\mathcal{E}),\alpha)}=P_{T(\mathcal{E})}$ and the inequality reads $rP_{T(\mathcal{E})}\leq0,$ which implies $T(\mathcal{E})=0.$
\end{proof}

\begin{lemma}
Suppose $\mathcal{E}$ is a torsion sheaf. If $(\mathcal{E}, \alpha)$ is semi-stable, then it is stable, which in turn is equivalent to the assertion that $\alpha$ is injective and $P_{\mathcal{E}}=\delta.$
\end{lemma}
\begin{proof}
Semi-stability for a non-trivial torsion sheaf requires $$P_{(\mathcal{E},\alpha)}=P_{(\mathcal{E})}-\delta=0.$$
\end{proof}

Now we look at homomorphism of framed sheaves:
\begin{definition}
A homomorphism $\phi:(\mathcal{E},\alpha)\to(\mathcal{E}',\alpha')$ of framed sheaves is a homomorphism of the underlying sheaves $\phi:\mathcal{E}\to\mathcal{E}'$ which preserves the framing up to a homothety, that is $\exists\lambda\in k$ such that $\alpha'=\lambda\circ\alpha'.$
\end{definition}

\begin{lemma}
The set $\Hom((\mathcal{E},\alpha),(\mathcal{E}',\alpha'))$ of homomorphisms of framed sheaves is a linear subspace of $\Hom(\mathcal{E},\mathcal{E}').$ If $\phi:(\mathcal{E},\alpha)\to(\mathcal{E}',\alpha')$ is an isomorphism, then the factor $\lambda$ in the definition above can be taken in $k^{\ast}.$ In particular, the isomorphism $\phi_{0}:=\lambda^{-1}\phi$ satisfies $\alpha'\circ\phi_{0}=\alpha.$
\end{lemma}

\begin{proof}
The first part of the lemma is very easy to check. If $\phi$ is an isomorphism then $(\alpha'\circ\phi)\circ\phi^{-1}=\lambda\cdot(\alpha\circ\phi^{-1})=\lambda\cdot(\lambda'\alpha)=\alpha.$ Hence $\lambda'=\lambda^{-1}$ and consequently $\lambda$ is invertible. 

\end{proof}

\begin{lemma}
If $(\mathcal{E},\alpha)$ and $(\mathcal{E}',\alpha')$ are stable with the same reduce Hilbert polynomial $p,$ then any nontrivial homomorphism $\phi:(\mathcal{E},\alpha)\to(\mathcal{E}',\alpha')$ is an isomorphism. Moreover, in this case $\Hom((\mathcal{E},\alpha),(\mathcal{E}',\alpha'))\cong k.$ If in addition $\alpha\neq0,$ or equivalently $\alpha'\neq0,$ then there is a unique isomorphism $\phi_{0}$ with $\alpha'\circ\phi_{0}=\alpha.$
\end{lemma}

\begin{proof}
Suppose $\phi:(\mathcal{E},\alpha)\to(\mathcal{E}',\alpha')$ is non-trivial. The image $\mathcal{F}:=im\phi$ inherits framings $\beta$ and $\beta'$ when considered as a quotient of $\mathcal{E}$ and as a submodule of $\mathcal{E}',$ respectively. If $\beta'\neq0$ then $\beta\neq0$ and $\beta'=\lambda\beta$ for some $\lambda\neq0.$ In any case, one has $$rk(\mathcal{F})p\leq P_{(\mathcal{F},\beta)}\leq P_{(\mathcal{F},\beta')}\leq rk( {\mathcal{F}})p.$$ Therefore equality holds at all places. By the stability assumption one has: $\mathcal{E}\cong\mathcal{F}\cong\mathcal{E}',$ $\alpha=\beta\circ\phi,$ $\beta'=\alpha',$ and $\beta$ and $\beta'$ differ by a nontrivial factor. 
Indeed, $\mathcal{F}\to\mathcal{E}'$ is injective as $(\mathcal{F}, \beta')$ is a subpair of $(\mathcal{E}',\alpha').$ Since  
$(\mathcal{E}',\alpha')$ is stable then $\mathcal{F}\to\mathcal{E}'$ must also be surjective, otherwise we would have a subpair with the same reduced Hilbert polynomial $p,$ therefore contradicting stability hypothesis, or else, $\beta'=0,$ in opposition to our assumption. A similar argument holds for the surjective morphism $\mathcal{E}\to\mathcal{F}'.$

Hence $\phi$ is an isomorphism of framed sheaves. In order to prove the last statement it is enough to show that $\aut(\mathcal{E},\alpha)=k\cdot id_{\mathcal{E}}.$ Suppose $\phi$ is an automorphism of $(\mathcal{E},\alpha).$ Choose a point $x\in supp(\mathcal{E})$ and let $\mu$ be an eigenvalue of $\phi$ restricted to the fibre $\mathcal{E}(x).$ Then $\phi-\mu\cdot id_{\mathcal{E}}$ is not surjective at $x$ and hence not an isomorphism, which implies $\phi-\mu\cdot id_{\mathcal{E}}=0.$
\end{proof}

The above lemma implies that in particular a stable framed sheaf is simple in the sense that its nontrivial endomorphisms are homotheties.

\begin{lemma}
If $\deg(\delta)\geq d,$ then in any semi-stable framed sheaf $(\mathcal{E},\alpha)$ the framing $\alpha$ is injective or zero. Conversely, if $\alpha$ is the inclusion homomorphism of a submodule $\mathcal{E}$ of $\mathcal{D}$ of positive rank, then $(\mathcal{E},\alpha)$ is stable.
\end{lemma}

\begin{proof}
Suppose that $(\mathcal{E},\alpha)$ is semi-stable framed sheaf with $\deg(\delta)\geq d,$ and that the framing $\alpha:\mathcal{E}\to\mathcal{D}$ is not injective. We have the following

$$\xymatrix{ K \ar[rd]_{\beta=0}\ar[r]& \mathcal{E}\ar[d]^{\alpha}\ar[r]^{\alpha}& \mathcal{D} \ar@{=}[d] \\ & \mathcal{D}\ar@{=}[r]&\mathcal{D} &}$$ 

where $K=\ker{\alpha}$ is torsion-free and the induced framing on $K$ is $\beta=0$ and consequently $\epsilon(\beta)=0.$ By the semi-stability of $(\mathcal{E},\alpha)$ we have that $P_{(K,0)}=P_{K}\leq P_{(\mathcal{E},\alpha)}=P_{\mathcal{E}}-\delta,$ thus $$\deg(\delta)\leq \deg(P_{\mathcal{E}}-P_{K})<d,$$ which is absurd.

For the second part of the lemma, we assume $\alpha$ is injective and there is a destabilizing pair $(K,\beta).$ Then one has the following commutative diagram

$$\xymatrix{ 0 \ar[r] & K \ar[d]_{0\neq\beta} \ar[r] & \mathcal{E} \ar[d]^{\alpha} \ar[r] & \mathcal{E} / K \ar[d]^{0} \ar[r] & 0 \\ 0 \ar[r] & \mathcal{D} \ar@{=}[r] & \mathcal{D} \ar[r]^{0} & \mathcal{D} \ar[r] & 0}$$  

The resulting quotient pair $( \mathcal{E} / K, \alpha'' )$ is such that $\alpha''=0;$ this is because $\mathcal{E} / K$ is a subsheaf of $ \mathcal{D} / K,$ since $\alpha$ is injective. Moreover, the Hilbert polynomial of the quotient pair $( \mathcal{E} / K, \alpha'' )$ satisfies $$P_{( \mathcal{E} / K, \alpha'' )}=P_{\mathcal{E} / K}\leq P_{(\mathcal{E},\alpha)}=P_{\mathcal{E}}-\delta.$$
Hence, $$\deg(\delta)\leq \deg(P_{\mathcal{E}}-P_{K})<d,$$ again, contradicting the hypothesis, and we are done

\end{proof}

The above lemma shows that the study of semi-stable framed sheaves reduces to the study of subsheaves of $\mathcal{D}$ if the degree of the polynomial $\delta$ is larger or equal to the dimension of the variety $X.$ Moreover if we assume that the Hilbert polynomial $P$ is a fixed integer, then the sheaves in question parameterize the Grothendieck $\Quot$ scheme $\Quot_{X}(\mathcal{D},P).$ For this reason we impose that the the polynomial $\delta$ has degree less than $d$ and write: 
\begin{equation}\label{delta-poly}
    \delta(m)=\delta_{1}\frac{m^{d-1}}{(d-1)!}+\delta_{2}\frac{m^{d-2}}{(d-2)!}+\cdots+\delta_{d},
\end{equation}
the first non-zero coefficient being positive.

In \cite{Huy}, the authors also introduce the notion of a slope for a framed sheaf $(\mathcal{E},\alpha)$ as $\mu(\mathcal{E},\alpha):=\frac{\deg\mathcal{E}-\epsilon(\alpha)\delta_{1}}{rk(\mathcal{E})},$ with respect to the first non vanishing coefficient of the parameter polynomial $\delta.$ This allows to introduce

\begin{definition}
A framed sheaf $(\mathcal{E},\alpha)$ of positive rank is said to be $\mu$-(semi-) stable with respect to $\delta_{1},$ if it has torsion-free kernel and if for all subsheaves $\mathcal{E}'$ with induced framing $\alpha'$ and rank $r'$ satisfying $0<r'<r,$ one has 
$$\mu(\mathcal{E}',\alpha')(\leq)\mu(\mathcal{E},\alpha)$$

\end{definition}

Remark that we have the following sequence of implications $$\xymatrix@C-1pc@R-1pc{*\txt{$\mu$-stability\hspace{0.1cm}} \ar@{=>}[rr]\ar@{=>}[d]& &*\txt{\hspace{0.1cm}$\mu$-semi-stability} \\ *\txt{stability\hspace{0.1cm}}\ar@{=>}[rr] & & *\txt{\hspace{0.1cm}semi-stability}\ar@{=>}[u]}$$

\bigskip\bigskip

\begin{lemma}
Let $\mathcal{F}\subset\mathcal{G}\subset\mathcal{E}$ be coherent sheaves and $\alpha$ a framing of $\mathcal{E}.$ Then the framing induced on $\mathcal{G}/\mathcal{F}$ as a quotient of $\mathcal{G}$ and as a submodule of $\mathcal{E}/\mathcal{G}$ agree.
\end{lemma}

\begin{proof}
We consider the following commutative diagram:

$$\xymatrix@C-1pc@R-1pc{& \mathcal{F}\ar@{=}[dd]\ar[rr]\ar[dr] && \mathcal{G}\ar[dd]\ar[rr]\ar[dr]^{\beta} && \mathcal{G}/\mathcal{F}\ar[dd]\ar[rd]&& \\
&&\mathcal{D}\ar[dd]\ar@{-->}[rr] && \mathcal{D}\ar[dd]\ar@{-->}[rr] && \mathcal{D}\ar[dd] & \\
& \mathcal{F} \ar[rr] \ar[dr]_{\alpha|_{\mathcal{F}}} && \mathcal{E}\ar[dr]_{\alpha}\ar[rr] && \mathcal{E}/\mathcal{F} \ar[rd] && \\
&& \mathcal{D}\ar@{-->}[rr] && \mathcal{D}\ar@{-->}[rr] && \mathcal{D} & 
}
$$
where $\beta:=\alpha|_{\mathcal{G}}.$ Suppose that $\alpha|_{\mathcal{F}}=0,$ then all the framings in the square 
$$
\xymatrix@C-1pc@R-1pc{ \mathcal{G}\ar[d]\ar[r] &  \mathcal{G}/\mathcal{F}\ar[d] \\ \mathcal{E}\ar[r] &  \mathcal{E}/\mathcal{F} }
$$ are induced by the framing on $\mathcal{E}/\mathcal{F}.$ On the other hand, if $\alpha|_{\mathcal{F}}\neq0$ then, the induced framings on $\mathcal{G}/\mathcal{F}$ and $\mathcal{E}/\mathcal{F}$ are both zero. So there is no equivocation.

\end{proof}
From this lemma it follows that subsheaves of a framed sheaf inherit naturally a framing.

\begin{proposition}
Let $(\mathcal{E},\alpha)$ be a semi-stable framed sheaf with reduced Hilbert polynomial $p.$ Then there is a filtration $$\mathcal{E}_{\bullet}\qquad0=\mathcal{E}_{0}\subset\mathcal{E}_{1}\subset\cdots\subset\mathcal{E}_{s}=\mathcal{E}$$ such that all the factors $gr_{i}(\mathcal{E}_{\bullet}):=\mathcal{E}_{i}/\mathcal{E}_{i-1}$ together with the induced framing $\alpha_{i}$ are stable with respect to $\delta$ with reduced polynomial $p.$ Any such filtration is called a Jordan-H\"older filtration of $(\mathcal{E},\alpha)$ The framed sheaf $$(gr(\mathcal{E}),gr(\alpha)):=\oplus_{i}(gr_{i}(\mathcal{E}_{\bullet}),\alpha_{i})$$ does not depend on the choice of the Jordan-H\"older filtration.
\end{proposition}

\begin{proof}
\cite[Proposition 1.13]{Huy1}
\end{proof}

This proposition motivates the following definition

\begin{definition}
Two semi-stable framed sheaves $(\mathcal{E},\alpha)$ and $(\mathcal{E}',\alpha')$ with reduced Hilbert polynomial $p$ are called $S-$equivalent, if their associated graded objects $(gr(\mathcal{E},gr(\alpha)))$ and $(gr(\mathcal{E}',gr(\alpha')))$ are isomorphic.
\end{definition}

\begin{definition}
A  family of framed sheaves parameterized by a noetherian scheme $T$ of finite type is a pair $(\mathfrak{E},\Phi)$ on the product $X\times T$ of a $T-$flat sheaf of $\mathcal{O}_{X\times T}-$modules and a morphism $\Phi:\mathfrak{E}\to p_{X}^{\ast}\mathcal{D}$ of sheaves of $\mathcal{O}_{X\times T}-$modules. Here $p_{X}:X\times T\to X$ is the natural projection on $X.$
\end{definition}

Furthermore, one can define the two functors $\mathfrak{M}_{\delta,P,\mathcal{D}}^{s}(\bullet)$ and $\mathfrak{M}_{\delta,P,\mathcal{D}}^{ss}(\bullet),$ for a fixed sheaf $\mathcal{D}$ and a fixed Hilbert polynomial $P-\epsilon(\alpha)\delta,$ from the category of noetherian schemes of finite type $\mathfrak{S}ch$ to the category of sets $\mathfrak{S}et$ defined as follows:
For a scheme $T\in\obj(\mathfrak{S}ch)$ one assigns the following set:$$\mathfrak{M}_{\delta,P,\mathcal{D}}^{s}(T)=\left\{[\mathfrak{E}, \Phi] \Large{\textbf{/}}\small{ \begin{array}{l} \bullet\hspace{0.2cm} (\mathfrak{E},\Phi) \textrm{ is a family framed sheaf on }X\times T,\hspace{0.1cm}\textrm{flat on } T; \\ \bullet\hspace{0.2cm}\textnormal{The fibre } (\mathfrak{E}\otimes k(t),\Phi_{t}:\mathfrak{E}\otimes k(t)\to\mathcal{D}) \textnormal{ at a point } t\in T \\ \textrm{is a stable framed sheaf on } X \textrm{ with Hilbert polynomial } \\ \hspace{2.5cm}P-\epsilon(\alpha)\delta \end{array}} \right\}$$ and $$\mathfrak{M}_{\delta,P,\mathcal{D}}^{ss}(T)=\left\{[\mathfrak{E}, \Phi] \Large{\textbf{/}}\small{ \begin{array}{l}  \bullet\hspace{0.2cm} (\mathfrak{E},\Phi) \textrm{ is a family framed sheaf on }X\times T,\hspace{0.1cm}\textrm{flat on } T; \\ \bullet\hspace{0.2cm}\textnormal{The fibre } (\mathfrak{E}\otimes k(t),\Phi_{t}:\mathfrak{E}\otimes k(t)\to\mathcal{D}) \textnormal{ at a point } t\in T \\ \textrm{is a semi-stable framed sheaf on } X \textrm{ with Hilbert} \\ \hspace{2cm} \textrm{ polynomial } P-\epsilon(\alpha)\delta \end{array}} \right\}.$$ $[\mathfrak{E},\Phi]$ stands for the class of the framed sheaf $(\mathfrak{E},\Phi)$; $(\mathfrak{F},\Psi)\in[\mathfrak{E},\Phi]$ if there is a line bundle $L$ on $T$ such that $\mathfrak{F}\cong\mathfrak{E}\otimes p_{X}^{\ast}L.$ We also denote by $k(t)$ the residue field of a point $t\in T.$ The framings $\Phi$ and $\Psi$ are compatible since at every point $t\in T$ they differ by an element $\lambda\in p^{\ast}L\otimes k(t)\cong k$. The main result for these objects is

\begin{theorem}\cite[Theorem 0.1]{Huy}\label{main}
Let $\delta\in\mathbb{Q}[m]$ be a polynomial with positive leading coefficient and of degree $\deg(\delta)<dim(X).$ Then there is a projective scheme $\mathcal{M}_{\delta,P,\mathcal{D}}^{ss}$ which is a coarse moduli space for the functor $\mathfrak{M}_{\delta,P,\mathcal{D}}^{ss}.$ Moreover there is an open subscheme $\mathcal{M}_{\delta,P,\mathcal{D}}^{s}$ which represent the functor $\mathfrak{M}_{\delta,P,\mathcal{D}}^{s}.$ A closed point in $\mathcal{M}_{\delta,P,\mathcal{D}}^{ss}$ represents an $S-$equivalence class of semi-stable framed sheaf.
\end{theorem}



\section{Framed pairs on $\bp3$:}\label{framedp3}


\subsection{Instantons on $\bp3$}

Let $\p3$ be the projective $3$-space, over the field of complex numbers $\mathbb{C},$ and fix a point $p_{0}$ in it. We will denote by $\bp3$ be the blow-up of $\p3$ at $p_{0},$ and by $\pi:\bp3\to\p3$ the blow-down map. Let $H,$ be the pull-back of the cohomology class of a hyperplane in $\ho^{2}(\p3,\mathbb{Z}),$ and set $E$ for the cohomology class of the exceptional divisor in $\ho^{2}(\bp3,\mathbb{Z}),$ then the Picard group of $\bp3$ can be written as ${\rm Pic}(\bp3)=H\mathbb{Z}\oplus E\mathbb{Z},$ 
and one can show that the Chow ring of $\bp3$ is isomorphic to the ring $$A^{\ast}(\bp3)=\frac{\mathbb{Z}[E,H]}{<E\cdot H, E^{3}-H^{3}>}.$$

For a divisor $D=a\cdot H+b\cdot E$ we associate the the line bundle $\mathcal{O}(D)=\mathcal{O}(a,b):=\mathcal{O}(a\cdot H)\otimes\mathcal{O}(b\cdot E).$ In particular, the canonical bundle, associated to the canonical divisor ${\rm K}_{\bp3}=-4H+2E,$ is denoted by $\omega_{\bp3}=\mathcal{O}(-4,2).$

As in \cite{henni2021} we shall adopt the following:

\begin{definition}\label{Instanton-def}
A rank $r$ torsion-free sheaf $\mathcal{F},$ on $\bp3,$ will be called an instanton sheaf, if it is $\mu_{L}$-semi-stable, for the polarization $L=\mathcal{O}(2,-1),$ has first Chern class $c_{1}(\mathcal{F})=0$ and satisfies:
\begin{itemize}
    \item[(i)] $\ho^{0}(\mathcal{F})=\ho^{3}(\mathcal{F}(-2,1))=0;$ 
    \item[(ii)] $\ho^{1}(\mathcal{F}(-2,1))=\ho^{2}(\mathcal{F}(-2,1))=0;$
    \item[(iii)] $\ho^{2}(\mathcal{F}(0,-1))=\ho^{2}(\mathcal{F}(-1,1))=0.$
\end{itemize}
\end{definition}

On $\bp3$, the definition of a rank $2$ locally-free instanton was given in \cite{ccgm} as generalization a mathematical instanton \cite{OS}, on $\p3,$ to the case of a fano $3$-fold with polycyclic Picard group. Previous notions of an instanton bundle (of rank $2$) for fano $3$-folds with Picard number $1$ has been introduced by Faenzi \cite{faenzi} and Kuznetsov \cite{kuznetsov}. The definition above has been given in \cite{henni2021} in order to include torsion-free instantons and higher rank analogues. We remind the reader that examples of such instanton bundles were obtained in \cite[Construction 7.2]{ccgm}, via Hartshorne-Serre correspondence, from a union of disjoint lines.  We shall call them \emph{t'Hooft instantons} on $\bp3,$ since they generalize the t'Hooft instantons on $\p3.$ The non-locally-free t'Hooft instantons are those obtained by elementary transformations from t'Hooft instanton bundles \cite[Section 4]{henni2021}. The charge of the instanton is given by degree of the second Chern class $c_{2}(\mathcal{E})=(k\cdot H^{2}+l\cdot E^{2})$, with respect to a chosen polarization and, in our case, it reads $$\deg_{L}(c_{2}(\mathcal{E}))=c_{2}(\mathcal{E})\cdot(2H-E)=(k\cdot H^{2}+l\cdot E^{2})\cdot(2H-E)=2k+l,$$ where $k,l$ are non negative integers and the product is taken in the Chow ring of $\bp3.$ We also recall the following results from \cite[Section 5 ]{henni2021}:
\begin{lemma}\label{stab-rest-H} Let $\mathcal{E}$ be an instanton bundle, on $\bp3,$ whose restriction $\mathcal{E}|_{E}$ to the exceptional divisor is $\mu$-semi-stable, then its restriction, $\mathcal{E}|_{S},$ to the generic element $S\in|H|,$ in $\bp3,$ is also $\mu$-semi-stable.
In particular, the above assertion is true for every t'Hooft instanton bundle on $\bp3.$ 
\end{lemma}
\begin{proof}
By considering the restriction sequence for $S$ twisted by $\mathcal{E}(-1,0)$ and using the fact that $E$ is $\mu_{L}$-stable, one gets the following exact sequence in cohomology:
$$0\to\ho^{0}(\mathcal{E}|_{S}(-1))\to\ho^{1}(\mathcal{E}(-2,0))\to\ho^{1}(\mathcal{E}(-1,0))\to\ho^{1}(\mathcal{E}|_{S}(-1))\to0.$$
$S$ is isomorphic to a plane in $\p3,$ and by the Hoppe criteria \cite[Ch.II, Lemma 1.2.5]{oss} it suffices to show that the utmost left term in the sequence above is zero. This is done by considering the sequence 
$$0\to\mathcal{E}(-2,0)\to\mathcal{E}(-2,1)\to\mathcal{E}|_{E}(-1)\to0.$$ Since $\mathcal{E}|_{E}$ is $\mu$-semi-stable \cite[\S 6]{ccgm} then $\ho^{0}(\mathcal{E}|_{E}(-1))$ and the map $$\ho^{1}(\mathcal{E}(-2,0))\to\ho^{1}(\mathcal{E}(-2,1))$$  is injective. Moreover, the instantonic condition $\ho^{1}\mathcal{E}(-2,1))=0$ implies that $\ho^{1}(\mathcal{E}(-2,0))$ is trivial and hence $\ho^{0}(\mathcal{E}|_{S}(-1))=0.$
 The last assertion follows from the fact that every t'Hooft instanton bundle is trivial on the the generic line, contained in the exceptional divisor \cite[Theorem 1.7]{ccgm}.  
\end{proof}

\begin{corollary}\label{triv-split-bdle}
An instanton bundle $\mathcal{E},$ on $\bp3,$ whose restriction $\mathcal{E}|_{E}$ to the exceptional divisor is $\mu$-semi-stable, is of trivial splitting type on the generic conic $C$ given by the pull-back of a generic line in $\p3.$
\end{corollary}
\begin{proof}
This follows from the $\mu$-semi-stability of $\mathcal{E}_{S},$ for a surface  $S\in |H|,$ containing $C,$ and the Grauert -- M\"ulich Thereom \cite[Ch.II, \S2, Corollary 2]{oss}
\end{proof}

\begin{corollary}\label{triv-split-shf}
Let $\mathcal{F}$ be an elementary transform of an instanton bundle $\mathcal{E},$ on $\bp3,$ whose restriction $\mathcal{E}|_{E}$ to the exceptional divisor is $\mu$-semi-stable then it is also of trivial splitting type on the generic conic $C$ given by the pull-back of a generic line in $\p3.$
\end{corollary}
\begin{proof}
The generic conic $C$ does not intersect the curve $\sigma,$ along which the elementary transform is realized, so $\mathcal{F}|_{C}\cong\mathcal{E}|_{C}=\mathcal{O}_{C}^{\oplus 2}.$ The last equality follows from Corollary \ref{triv-split-bdle}.   
\end{proof}


\subsection{Moduli of framed instanton pairs}

\begin{definition}
A framed sheaf $\mathfrak{E}:=(\mathcal{E},\alpha)$ will be called  framed instanton sheaf if $\mathcal{E}$ is an instanton sheaf.  
\end{definition}

In the following, we consider {\em framed t'Hooft instanton sheaves} $\mathfrak{E}:=(\mathcal{E},\alpha),$ where $\mathcal{E}$ is t'Hooft instanton sheaf, that is, an instanton corresponding to a union of lines via Hartshorne-Serre correspondence or any elementary transformation of it \cite[\S4]{henni2021}. We fix $\mathcal{E}_{D}$ a rank two torsion-free instanton sheaf on $D\in |\mathcal{O}_{\bp3}(1,0)|$ and let $\alpha:\mathcal{E}\to\mathcal{E}|_{D}=:\mathcal{E}_{D},$ be the framing on the divisor $D\cong\p2.$



\begin{lemma}\label{stab}
    A rank $2$ framed t'Hooft instanton sheaf $\mathfrak{E}=(\mathcal{E},\alpha)$ is $\mu_{L,\delta_{1} }$-stable for some positive $\delta_{1}.$
\end{lemma}

\begin{proof}
    Let $\mathfrak{F}:=(\mathcal{F}, \beta)$ a framed subsheaf of $\mathfrak{E}=(\mathcal{E},\alpha)$ of rank $1,$ that is the subsheaf $\mathcal{F}\stackrel{f}{\hookrightarrow}\mathcal{E}$ has rank $1.$ $\mathcal{F}$ is obviously torsion-free since $\mathcal{E}$ is. By the smoothness of $\bp3$ it clearly  follows that $\mathcal{F}$ is a line bundle, that is, $\mathcal{F}$ is of the form $\mathcal{O}_{\bp3}(p,q)$ for some $p,q\in\mathbb{Z}.$  In order to prove stability, we shall contemplate two cases
 
    
    
   
 \begin{itemize}
       \item[(i)] If $\mathcal{F}\not\subset\ker{\alpha}:$ 
    The $\mu_{\delta_{1},L}$-stability condition reads $$\mu_{\delta_{1},L}(\mathcal{O}_{\bp3}(p,q))-\delta_{1}=4p+q-\delta_{1}\leq \mu_{\delta_{1},L}(\mathcal{E})-\frac{\delta_{1}}{2}=-\frac{\delta_{1}}{2}.$$
    Hence  \begin{equation}\label{ineq1}4p+q\leq\frac{\delta_{1}}{2}.\end{equation} 
    On the other hand $\mathcal{E}$ is $\mu_{L}$-semi-stable and from the Hoppe criteria one has $4p+q\leq0$ for all line bundles $\mathcal{O}_{\bp3}(p,q)\hookrightarrow\mathcal{E}.$
Thus \eqref{ineq1} is clearly satisfied for all $\delta_{1}>0.$
\vspace{0.5cm}
      \item[(ii)] If $\mathcal{F}\subset\ker{\alpha}=\mathcal{E}(-1,0):$ In this case $\mathcal{O}_{\bp3}(p,q)\hookrightarrow\mathcal{E}(-1,0),$ and equivalently, $$\mathcal{O}_{\bp3}(p+1,q)\hookrightarrow\mathcal{E},$$ The $\mu_{\delta_{1},L}$-stability condition in this case $(\epsilon(\beta)=0)$ reads: 
      $$\mu_{\delta_{1},L}(\mathcal{O}_{\bp3}(p+1,q))=4p+q+4\leq \mu_{\delta_{1},L}(\mathcal{E})-\frac{\delta_{1}}{2}=-\frac{\delta_{1}}{2},$$
    and one has \begin{equation}\label{ineq2}4p+q\leq-\frac{\delta_{1}}{2}-4.\end{equation} 
    Again, by $\mu_{L}$-semi-stability of $\mathcal{E},$ and the Hoppe criteria, it follows that $4p+q\leq-4$ for all line bundles $\mathcal{O}_{\bp3}(p+1,q)\hookrightarrow\mathcal{E},$ and \eqref{ineq2} is satisfied for all $0<\delta_{1}<8.$

From both cases, the result follows for $0<\delta_{1}<8.$
    \end{itemize}
\end{proof}


\begin{proposition}
    Fix a torsion-free instanton sheaf $\mathcal{E}_{D}$ of rank $2$ on $D (\cong\p2)\in |\mathcal{O}_{\bp3}(1,0)|$ and let $c=r-[k\cdot H^{2}+l\cdot E^{2}]\in\ho^{\ast}(\bp3,\mathbb{Q}).$ Then there exists a quasi-projective scheme $\mathfrak{U}_{\delta, n}(c)$ that is a fine moduli for locally free framed t'Hooft pairs $(\mathcal{E},\mathcal{E}\stackrel{\alpha}{\to}\mathcal{E}_{D}),$ for $0<\delta_{1}<8.$ 
    
    
\end{proposition}

\begin{proof}
    By Theorem \ref{main}, for any $c\in\ho^{\ast}(\bp3,\mathbb{Q})$ there is a quasi-projective scheme $\mathfrak{M}_{X}(c),$ that represents framed pairs $(\mathcal{E},\mathcal{E}\stackrel{\alpha}{\to}\mathcal{E}_{D}),$ for ${\rm ch}(\mathcal{E})=c$ for some $\delta\geq0$ and $n\in\mathbb{N}.$ It follows from Lemma \ref{stab} that for $c=r-[k\cdot H^{2}+l\cdot E^{2}]$ there is a subscheme of $\mathfrak{M}_{X}(c)$ representing framed t'Hooft instantons, when $0<\delta_{1}<8.$ Furthermore, by semi-continuity; this locus is open since such framed bundles are given by vanishing cohomological conditions (Definition \ref{Instanton-def}) and the additional condition of trivial splitting on the conic $\mathcal{C}$ given by pulling back a line in $\p3$ via the blow-down map (Corollary \ref{triv-split-bdle}). As opens subsets of a quasi-projective scheme are also quasi-projective, the result follows.     
\end{proof}

At last, we shall prove the following 

\begin{theorem}
The fine moduli (quasi-projective) scheme $\mathfrak{U}_{\delta, n}(c)$ that represents locally free framed t'Hooft pairs is unobstructed 

\end{theorem}

\begin{proof}

In the deformation study of framed modules \cite[\S4]{Huy}, the obstruction to the smoothness of the moduli space at a point represented by the $(\mathcal{E},\alpha)$ is given by the hyper-ext group $\mathbb{E}{\rm xt}^{2}(\mathcal{E}, \mathcal{E}\stackrel{\alpha}{\to}\mathcal{E}_{D})$ for the complexes $\mathcal{E},$ in degree zero and $\mathcal{E}\stackrel{\alpha}{\to}\mathcal{E}_{D},$ in degrees $0$ and $1.$ In our case, $\mathcal{E}_{D}$ is locally free along the divisor $D\in |\mathcal{O}_{\bp3}(-1,0)|$ and the latter is quasi-ismorphic to $\mathcal{E}(-1,0)$ as a complex in degree zero, Hence 
$$\mathbb{E}{\rm xt}^{2}(\mathcal{E}, \mathcal{E}\stackrel{\alpha}{\to}\mathcal{E}_{D})\cong \ext^{2}(\mathcal{E},\mathcal{E}(-1,0))\cong\ho^{2}(\mathcal{E}nd({E})(-1,0)).$$ 
By construction \cite[Construction 7.2]{ccgm} $\mathcal{E}$ fits the exact sequence 
\begin{equation}\label{const-seq}
    0\to\mathcal{O}_{\bp3}(-1,1)\to\mathcal{E}\to\mathcal{I}_{X}(1,-1)\to0.
\end{equation}
Where $\mathcal{I}_{X}$ is the ideal sheaf of a certain disjoint union of lines, the interested reader might find more details, about $X,$ in \cite[Construction 7.2]{ccgm}. One can easily check that the restriction $\mathcal{E}(-1,0)|_{X_{i}},$ on each irreducible component $X_{i}$ of $X,$ is either $\mathcal{O}_{\p1}(1)^{\oplus 2}$ or $\mathcal{O}_{\p1}(1)\oplus\mathcal{O}_{\p1}$ hence $\ho^{1}(D,\mathcal{E}(-1,0)|_{X})=0$

Twisting the sequence \ref{const-seq} by  $\mathcal{O}_{\bp3}(-1,0)$ and applying the functor $\Hom(\mathcal{E})$ yields the long exact sequence $$\to\ho^{2}(\mathcal{E}(-2,1))\to\ho^{2}(\mathcal{E}nd({E})(-1,0))\to\ho^{2}(\mathcal{E}\otimes\mathcal{I}_{X}(0,-1))\to$$

As $\ho^{2}(\mathcal{E}(-2,1))$ is zero for an instanton sheaf and the right hand term fits into the sequence 
$$\to0=\ho^{1}(D,\mathcal{E}(-1,0)|_{X})\to\ho^{2}(\mathcal{E}\otimes\mathcal{I}_{X}(0,-1))\to\ho^{2}(\mathcal{E}(0,-1))=0,$$ the vanishing of the obstruction follows. 


\end{proof}




\end{document}